\documentclass[10pt,dvipsnames]{amsart}

\usepackage[utf8]{inputenc}
\usepackage{lmodern}
\usepackage[T1]{fontenc}
\usepackage[english]{babel}
\usepackage{fullpage}

\usepackage{latexsym}
\usepackage{amsmath,amssymb,amsthm}
\usepackage{mathrsfs}
\usepackage{stmaryrd}

\usepackage{enumerate}
\usepackage{enumitem}

\usepackage{rotating}

\usepackage{todonotes}
\usepackage{tikz-cd}
\usetikzlibrary{shapes}
\usepackage{tikz}

\makeatletter
\DeclareRobustCommand{\rvdots}{%
	\vbox{
		\baselineskip4\p@\lineskiplimit\z@
		\kern-\p@
		\hbox{.}\hbox{.}\hbox{.}
}}
\makeatother

\usepackage{xcolor}
\usepackage{color}
\usepackage{hyperref}
\hypersetup{
	colorlinks=true,
	linkcolor=blue,
	citecolor=blue,
	urlcolor=blue
}

\usepackage{ulem}
\usepackage{soul}

\newtheorem{theoremAlph}{Theorem}
\newtheorem{corollaryAlph}[theoremAlph]{Corollary}

\newtheorem{theorem}{Theorem}[section]
\newtheorem{lemma}[theorem]{Lemma}	
\newtheorem{proposition}[theorem]{Proposition}

\theoremstyle{definition}
\newtheorem{definition}[theorem]{Definition} 
\newtheorem{remark}[theorem]{Remark}

\theoremstyle{definition} 
\newtheorem*{ack}{Acknowledgements}

\numberwithin{equation}{section}

\newcommand{\C}{\mathbb{C}} 
\newcommand{\R}{\mathbb{R}} 
\newcommand{\Z}{\mathbb{Z}} 
\newcommand{\N}{\mathbb{N}} 
\newcommand{\Quat}{\mathbb{H}}

\newcommand{\II}{\mathrm{I\!I}}
\newcommand{\Ric}{\textup{Ric}}

\newcommand{\ttimes}{\mathrel{\widetilde{\times} }}

\makeatletter
\DeclareRobustCommand*\uell{\mathpalette\@uell\relax}
\newcommand*\@uell[2]{
	\setbox0=\hbox{$#1\ell$}
	\setbox1=\hbox{\rotatebox{10}{$#1\ell$}}
	\dimen0=\wd0 \advance\dimen0 by -\wd1 \divide\dimen0 by 2
	\mathord{\lower 0.1ex \hbox{\kern\dimen0\unhbox1\kern\dimen0}}
}

\makeatletter
\newcommand{\mylabel}[2]{#2\def\@currentlabel{#2}\label{#1}}
\makeatother

\begin{document}
	\title[\rmfamily Examples of tangent cones of non-collapsed Ricci limit spaces]{\rmfamily Examples of tangent cones of non-collapsed Ricci limit spaces}
	\date{\today}
	\author{Philipp Reiser}
	\address{Department of Mathematics, University of Fribourg, Switzerland}
	\email{\href{mailto:philipp.reiser@unifr.ch}{philipp.reiser@unifr.ch}}
	
	\normalem
	
	\begin{abstract}
		We give new examples of manifolds that appear as cross sections of tangent cones of non-collapsed Ricci limit spaces. It was shown by Colding--Naber that the homeomorphism types of the tangent cones of a fixed point of such a space do not need to be unique. In fact, they constructed an example in dimension 5 where two different homeomorphism types appear at the same point. In this note, we extend this result and construct limit spaces in all dimensions at least 5 where any finite collection of manifolds that admit \emph{core metrics}, a type of metric introduced by Perelman and Burdick to study Riemannian metrics of positive Ricci curvature on connected sums, can appear as cross sections of tangent cones of the same point.
	\end{abstract}

	\maketitle

	\section{Introduction and Main Results}
	
	In this note, we consider pointed Gromov--Hausdorff limits $(Y,d_Y,y)$ of sequences of pointed $n$-dimensional Riemannian manifolds $(M_i,g_i,x_i)$ with a lower Ricci curvature bound, i.e.\
	\[ \Ric(g_i)\geq -(n-1) \]
	for all $i$. Additionally, we require that the limit is non-collapsed, i.e.\ there exists $v>0$ such that
	\[ \textrm{vol}(B_1(p_i))\geq v \]
	for all $i$. We call such a space $(Y,d_Y,p)$ a \emph{non-collapsed Ricci limit space}. The structure of non-collapsed Ricci limit spaces, or Ricci limit spaces in general, has been studied extensively, see e.g.\ \cite{BPS24}, \cite{Ch01}, \cite{CC97,CC00,CC00a}, \cite{CJN21}, \cite{CN13a,CN15}, \cite{Co97}, \cite{CN13}, and the references therein.
	
	When studying the structure of the limit space $Y$, a central role is played by its \emph{tangent cones}. A tangent cone at $x\in Y$ is the pointed Gromov--Hausdorff limit of a sequence $(Y,R_i^{-1}d_Y,x)$, where $R_i\to 0$ as $i\to\infty$. By Gromov's precompactness theorem, any such sequence has a converging subsequence. Moreover, if $Y$ is non-collapsed, it was shown by Cheeger--Colding \cite{CC97} that every tangent cone of $Y$ is a metric cone, and it follows from work of Ketterer \cite{Ke15} that the cross section of the cone satisfies the curvature dimension condition $\mathrm{CD}(n-2,n-1)$.
	
	However, even in the non-collapsed case, the tangent cones at a point obtained from different sequences $R_i$ do not need to coincide. This was first demonstrated by Perelman \cite{Pe97a} and Cheeger--Colding \cite{CC97}, where they constructed a family of metrics on $S^3$ whose cones all appear as tangent cones of the same point of a non-collapsed Ricci limit space. Subsequently, Colding--Naber \cite{CN13} gave further examples, including cones over $S^{n-1}$ that isometrically split off precisely $\R^k$ for all $0\leq k\leq n-2$, and cones whose cross sections are not even homeomorphic. The latter was realized by a $5$-dimensional limit space that contains a point with two tangent cones whose cross sections are given by $S^4$ and $\C P^2\#(-\C P^2)$, respectively. We note that this is in strong contrast to the situation of non-collapsed limits of manifolds with a lower sectional curvature bound, where all tangent cones are unique (see \cite{Pe91}, \cite{Ka07}) and all spaces of directions are homeomorphic to spheres (see \cite{Ka02}).
	
	The goal of this note is to provide further examples of non-collapsed Ricci limit spaces with non-homeomorphic tangent cones. To formulate our main result, we first need to recall the notion of \emph{core metrics}.
	
	\begin{definition}
		A Riemannian metric $g$ of positive Ricci curvature on an $n$-dimensional smooth manifold $M$ is a \emph{core metric} if there exists an embedded disc $D^n\subseteq M$ such that the induced metric on the boundary of $M\setminus{D^n}^\circ$ is the round metric on $S^{n-1}$ and its second fundamental form is strictly positive.
	\end{definition}
	
	Based on work by Perelman \cite{Pe97}, core metrics were introduced by Burdick \cite{Bu19} in the context of preserving positive Ricci curvature along connected sums. The known examples of manifolds with core metrics are given as follows:
	\begin{enumerate}
		\item The sphere $S^n$ and the compact rank one symmetric spaces $\C P^n$, $\mathbb{H}P^n$ and $\mathbb{O}P^2$ (see \cite{Bu19}, \cite{Pe97}),
		\item Linear sphere bundles and projective bundles with fibre $\C P^n$, $\mathbb{H}P^n$ or $\mathbb{O}P^2$ over manifolds with core metrics (see \cite{Bu20}, \cite{Re23,Re24a}),
		\item Products of manifolds with core metrics (see \cite{Re24a}),
		\item Connected sums of manifolds with core metrics (see \cite{Bu20a}),
		\item Manifolds obtained as boundaries of certain plumbings (see \cite{Bu19a}, \cite{Re23}),
		\item Certain manifolds that decompose as the union of two disc bundles, such as the Wu manifold $W^5$ (see \cite{Re24a}).
	\end{enumerate}
	Note also that a closed manifold that admits a core metric is simply-connected by \cite{La70}.
	
	Our main result is the following.
	
	\begin{theoremAlph}\label{T:core_tangent_cone}
		Let $M_1^n,\dots,M_\ell^n$ be closed, smooth, $n$-dimensional manifolds that admit core metrics. Then there exist a non-collapsed Ricci limit space $(Y^{n+1},d_Y,y)$ and (non-smooth) metrics $d_i$ on each $M_i$ such that the cones $C(M_i,d_i)$ all are tangent cones of $Y$ at $y$.
	\end{theoremAlph}
	Each of the metrics $d_i$ is in fact a Riemannian metric with $(2\ell-1)$ singular points which isometrically contains the core metric on $M_i$ with the embedded disc removed. The limit space $Y$ is homeomorphic to the cone over the connected sum
	\[M=M_1\#\dots\# M_\ell\#(-M_1)\#\dots\#(-M_\ell)\]
	equipped with a Riemannian metric with a singularity at the origin. The construction of the metric is based on a technique developed by Colding--Naber \cite{CN13} to construct metrics with non-unique tangent cones on a (topological) cone. One of the main requirements of this technique is the construction of a \emph{Ricci closable} metric (see Definition \ref{D:Ricci_closably} below) on the cross section. To prove Theorem \ref{T:core_tangent_cone}, we will give a general criterion for the existence of a Ricci closable metric on a manifold with an isometric $\Z/2$-action and apply it to the $\Z/2$-action on $M$ that interchanges each $M_i$ with $(-M_i)$.
	
	Theorem \ref{T:core_tangent_cone} shows in particular that for any closed manifold $M$ that admits a core metric there exists a (non-smooth) metric $d$ such that the cone $C(M,d)$ is a non-collapsed Ricci limit space, since every tangent cone of the limit space is the limit of an appropriate rescaling of the original sequence. In general, it is open on which closed manifold $M$ there exists a metric $d$ such that the cone $C(M,d)$ is a non-collapsed Ricci limit space. For example, it was shown in \cite{Zh24} that this is the case for all 3-dimensional spherical space forms. Moreover, we give further examples in Section \ref{S:Ricci_clos_ex} below (which, in contrast to the manifolds in Theorem~\ref{T:core_tangent_cone}, may also be non-simply-connected). On the other hand, by \cite{ST21} or \cite{BPS24} there is no such metric on $\R P^2$. Further obstructions are given in \cite{Ho17} if one additionally assumes that the converging sequence is orientable. This is in contrast to the collapsed case, where it follows from the work of Sha--Yang \cite{SY91} that for every closed Riemannian $n$-manifold $(M,g)$ of Ricci curvature at least $(n-1)$ the cone $C(M,g)$ is a collapsed Ricci limit space, see Theorem \ref{T:tangent_cone_collapsed} below. By using Theorem \ref{T:core_tangent_cone}, we obtain a similar statement in the non-collapsed case when restricting to simply-connected 4-manifolds.
	
	\begin{corollaryAlph}\label{C:dim-4}
		Let $M^4$ be a closed, smooth, simply-connected $4$-manifold that admits a Riemannian metric of positive scalar curvature. Then there exists a (non-smooth) metric $d$ on $M$ such that the cone $C(M,d)$ is a non-collapsed Ricci limit space. Moreover, if $M$ is the boundary of a compact, oriented, smooth $5$-manifold, then, after possibly changing the smooth structure on $M$, we can assume that $d$ is a smooth Riemannian metric.
	\end{corollaryAlph}
	

	This article is organised as follows. In Section \ref{S:prel} we recall the construction of \cite{CN13} to construct non-collapsed Ricci limit spaces with prescribed tangent cones, as well as basic results on Ricci curvature. Further, in Section \ref{S:Ricci_clos_crit} we establish a criterion for a given Riemannian metric to be Ricci closable and apply this to prove Theorem \ref{T:core_tangent_cone}. In Section \ref{S:Ricci_clos_ex} we give further examples of Ricci closable metrics that are based on known construction methods for Riemannian metrics of positive Ricci curvature. Finally, in Appendix \ref{S:collapsed} we consider the collapsed case and recall the construction of \cite{SY91}.
	
	\begin{ack}
		The author would like to thank Christian Ketterer and Daniele Semola for helpful comments and suggestions on an earlier version of this article.
	\end{ack}

	\section{Preliminaries}\label{S:prel}
	
	
	Following \cite{CN13}, for a non-collapsed Ricci limit space $(Y,d_Y,y)$ we define $\overline{\Omega}_{Y,y}$ as the family of metric spaces $\{(X_s,d_s)\}$ such that $C(X_s,d_s)$ is a tangent cone of $Y$ at $y$.
	
	\begin{definition}[\cite{CN13}]\label{D:Ricci_closably}
		A Riemannian manifold $(M^n,g)$ is \emph{Ricci closable} if for every $\varepsilon>0$ there exists a pointed open Riemannian manifold $(N^{n+1}_\varepsilon,h_\varepsilon,y_\varepsilon)$ of non-negative Ricci curvature such that the annulus $A_{1,\infty}(y_\varepsilon)\subseteq N_\varepsilon$ is isometric to the annulus $A_{1,\infty}(o)\subseteq C(M,(1-\varepsilon)g)$.
	\end{definition}
	Note that a Ricci closable Riemannian manifold in particular has positive Ricci curvature.
	
	The main result of \cite{CN13} is now given as follows.
	
	\begin{theorem}[{\cite[Theorem 1.1]{CN13}}]\label{T:tang_cone_crit}
		Let $\Omega$ be a connected manifold and $X^{n-1}$ a closed manifold with $n\geq 3$. Let $\{g_s\}_{s\in\Omega}$ be a smooth family of Riemannian metrics on $X$ such that
		\begin{enumerate}
			\item $\mathrm{Vol}(X,g_s)=V\leq \mathrm{Vol}(S^{n-1},ds_{n-1}^2)$,
			\item $\Ric_{g_s}\geq n-2$,
			\item There exists $s_0\in\Omega$ such that $g_{s_0}$ is Ricci closable.
		\end{enumerate}
		Then there exists a non-collapsing sequence of pointed complete Riemannian manifolds $(M^n_\alpha,g_\alpha,p_\alpha)$ of $\Ric\geq -(n-1)$ that Gromov--Hausdorff converges to a pointed metric space $(Y,d_Y,y)$ with $\overline{\Omega}_{Y,y}=\overline{\{(X,g_s)\}}$.
	\end{theorem}
	
	We will also need the following gluing result of Perelman \cite{Pe97}.
	
	\begin{theorem}[{\cite[Section 4]{Pe97}, see also \cite[Section 2]{BWW19}}]\label{T:gluing}
		Let $(M_1,g_1)$, $(M_2,g_2)$ be Riemannian manifolds of $\Ric>0$ with compact boundaries such that there exists an isometry $\phi\colon \partial M_1\to \partial M_2$. Assume that the sum of second fundamental forms $\II_{\partial M_1}+\phi^*\II_{\partial M_2}$ is non-negative. Then the $C^0$-metric $g_1\cup_\phi g_2$ on $M_1\cup_\phi M_2$ can be smoothed in an arbitrarily small neighbourhood of the gluing area into a smooth metric of $\Ric>0$.
	\end{theorem}
	
	Here we use the convention that the second fundamental form of the boundary of a Riemannian manifold $(M,g)$ is given by
	\[ \II(u,v)=g(\nabla_u\nu,v), \]
	where $u,v\in T\partial M$ and $\nu$ is the outward unit normal field of $\partial M$. We say the boundary is \emph{convex} if $\II$ is positive definite.
	
	Finally, we recall the following formulae for the Ricci curvatures of a metric on a cylinder, see e.g.\ \cite[Lemma 2.1]{Re24a}
	\begin{lemma}
		\label{L:curv_form}
		Let $M$ be a manifold, $I$ an interval and let $g=dt^2+h_t$ be a Riemannian metric on $I\times M$, where $h_t$ is a smooth family of metrics on $M$. Let $h_t'$ and $h_t''$ denote the first and second derivative of $h_t$ in $t$-direction, respectively. Then the second fundamental form of a slice $\{t\}\times M$ with respect to the normal vector $\partial_t$ is given by
		\[\II=\frac{1}{2}h_t', \]
		and the Ricci curvature of the metric $g$ are given as follows:
		\begin{align*}
			&\Ric(\partial_t,\partial_t)=-\frac{1}{2}\mathrm{tr}_{h_t}h_t^{\prime\prime}+\frac{1}{4}\lVert h_t^\prime\rVert_{h_t}^2,\\
			&\Ric(v,\partial_t)\;=-\frac{1}{2}v(\mathrm{tr}_{h_t}h_t')+\frac{1}{2}\sum_{i}(\nabla^{h_t}_{e_i}h_t')(v,e_i),\\
			&\Ric(u,v)\;\;=\Ric^{h_t}(u,v)-\frac{1}{2}h_t^{\prime\prime}(u,v)+\frac{1}{2}\sum_i h_t^\prime(u,e_i)h_t^\prime(v,e_i)-\frac{1}{4}h_t^\prime(u,v)\mathrm{tr}_{h_t}h_t^\prime.
		\end{align*}
		Here $u,v\in T_xM$ and $\{e_i\}$ is an orthonormal basis of $T_xM$ with respect to $h_t$.
	\end{lemma}
	In the special case of a doubly warped product metric, we obtain the following.
	\begin{lemma}\label{L:doubly_warped_Ric}
		Let $(M_1^{n_1},g_1)$ and $(M_2^{n_2},g_2)$ be Riemannian manifolds and let $f_1,f_2\colon[0,t_0]\to(0,\infty)$ be smooth functions for some $t_0>0$. Then the Ricci curvatures of the metric
		\[ dt^2+f_1(t)^2g_1+f_2(t)^2g_2 \]
		on $[0,t_0]\times M_1\times M_2$ are given by
		\begin{align*}
			\Ric(\partial_t,\partial_t)&=-n_1\frac{f_1''}{f_1}-n_2\frac{f_2''}{f_2},\\
			\Ric(\tfrac{v_1}{f_1},\tfrac{v_1}{f_1})&=-\frac{f_1''}{f_1}+\frac{\Ric^{g_1}(v_1,v_1)-(n_1-1){f_1'}^2}{f_1^2}-n_2\frac{f_1'f_2'}{f_1f_2},\\
			\Ric(\tfrac{v_2}{f_2},\tfrac{v_2}{f_2})&=-\frac{f_2''}{f_2}+\frac{\Ric^{g_2}(v_2,v_2)-(n_2-1){f_2'}^2}{f_2^2}-n_1\frac{f_1'f_2'}{f_1f_2}\\
			\Ric(\partial_t,v_1)&=\Ric(\partial_t,v_2)=\Ric(v_1,v_2)=0,
		\end{align*}
		where $v_i\in TM_i$ are unit vectors with respect to $g_i$. Further, the second fundamental form of a slice $\{t\}\times M_1\times M_2$ with respect to the unit normal $\partial_t$ is given by
		\[\II(\tfrac{v_i}{f_i},\tfrac{v_j}{f_j})=\frac{f_i'}{f_i}\delta_{ij}. \]
	\end{lemma}

	\section{A criterion for Ricci closability}\label{S:Ricci_clos_crit}
	
	In this section we give a criterion for a Riemannian metric to be Ricci closable (Proposition \ref{P:Ricci_cl_crit}) and use this to prove Theorem \ref{T:core_tangent_cone}. We first slightly reformulate the condition of Ricci closability.
	
	\begin{lemma}\label{L:Ricci_cl_crit}
		Let $(M^n,g)$ be a closed Riemannian manifold of positive Ricci curvature and suppose that there exists a compact Riemannian manifold $(N,g_N)$ with convex boundary isometric to $(M,g)$. Then there exists $c>0$ such that $(M,c^2g)$ is Ricci closable.
	\end{lemma}
	\begin{proof}
		For $c>0$ we attach a cylinder $([0,\infty)\times M,dt^2+f(t)^2g)$ to $N$, where $f\colon[0,\infty)\to(0,\infty)$ is a smooth function satisfying $f'(0)=2c$, $f''|_{[0,1)}<0$ and $f'|_{[1,\infty)}\equiv c$. Then, for $c$ sufficiently small, the metric $dt^2+f(t)^2g$ has positive Ricci curvature near the boundary $\{0\}\times M$ and the sum of second fundamental forms of $g_N$ and $dt^2+f(t)^2g$ is positive on $M$ by Lemma \ref{L:doubly_warped_Ric}. Hence, we can apply Theorem \ref{T:gluing} to smooth this metric in a small neighbourhood of the gluing area while preserving positive Ricci curvature in this neighbourhood. Outside a neighbourhood of the manifold $N$, this metric is then of the form $dt^2+(ct+c_0)^2g$. Thus, by appropriately rescaling this metric we obtain that the metric $c^2g$ is Ricci closable.
	\end{proof}
	
	Our criterion is now given as follows.
	\begin{proposition}\label{P:Ricci_cl_crit}
		Let $(M^n,g)$ be a closed Riemannian manifold of positive Ricci curvature that admits an isometric action of $\Z/2$ such that the fixed point set $Y\subseteq M$ is a non-empty hypersurface with trivial normal bundle. Then there exists $c>0$ such that $(M,c^2 g)$ is Ricci closable.
	\end{proposition}
	\begin{proof}
		Since $Y$ is the fixed point set of an isometric action, it is totally geodesic. Moreover, since the normal bundle of $Y$ is trivial, cutting $M$ along $Y$ results in a manifold with two totally geodesic boundary components isometric to $Y$. By \cite{La70} these two boundary components cannot lie in the same connected component of $\overline{M\setminus Y}$, so $M$ decomposes into $M_1\cup_Y M_2$, where $M_1,M_2\subset M$ are submanifolds with common boundary $Y$, and the action of $\Z/2$ yields an isometry $\phi\colon(M_1,g|_{M_1})\to(M_2,g|_{M_2})$.
		
		We now slightly deform the metric $g_1=g|_{M_1}$ near $\partial M_1=Y$ as follows. First, we rescale the metric $g$ so that the Ricci curvatures of the induced metric on $Y$ are bounded from below by $-(n-2)$. A neighbourhood of the boundary $\partial M_1$ can be identified with $[0,\varepsilon]\times Y$ such that the metric $g_1$ is of the form $dt^2+h_t$ on this part, where $h_t$ is a smoothly varying family of metrics on $Y$. By Lemma \ref{L:curv_form}, since $Y$ is totally geodesic, we have $\partial_t h_t=0$ at $t=0$. 
		
		Now let $f\colon[0,\varepsilon]\to(0,\infty)$ be a smooth function satisfying
		\begin{equation}
			\label{EQ:f_ineq}
			-\frac{f''}{f}-(n-2)\frac{1+{f'}^2}{f^2}>0,
		\end{equation}
		and such that $f(0)=1$ and $f$ is an even function at $t=0$. Such a function can for example be obtained as the solution of the initial value problem
		\[ f''=-f-(n-2)\frac{1+{f'}^2}{f},\quad f(0)=1,\,f'(0)=0, \]
		where we possibly need to choose $\varepsilon$ smaller to ensure that a solution exists, and subsequently smoothing the function $t\mapsto f(|t|)$ on $[-\varepsilon,\varepsilon]$ at $t=0$ to obtain an odd function at $t=0$. Since the original function is already $C^2$ at $t=0$, the smoothing can be done $C^2$-close, so that \eqref{EQ:f_ineq} is still satisfied.
		
		By Lemma \ref{L:doubly_warped_Ric}, the metric $dt^2+f(t)^2h_0$ on $[0,\varepsilon]\times Y$ has positive Ricci curvature. Hence, since the $1$-jets of the metrics $g_1$ and $dt^2+f(t)^2h_0$ coincide on $Y$, we can apply the deformation result of \cite{Wr02}, see also \cite{BH22}, to deform $g_1$ through metrics $g_s$, $s\in[0,1]$, of positive Ricci curvature into a metric $g_0$ that is of the form $dt^2+f(t)^2h_0$ on $[0,\varepsilon']\times Y$ for some $\varepsilon'\in(0,\varepsilon]$. Moreover, since near $Y$ the metric $g_s$ defined in \cite{Wr02} is a convex combination of the metrics $g_0$ and $g_1$, and since both the metric $g_0$ and $g_1$ define a smooth metric on the double $M_1\cup_Y M_1$, the same holds true for all metrics $g_s$.
		
		Next, we consider the space $N=I\times M_1$, where $I=[0,\pi]$. Then
		\[ \partial N\cong M_1\cup_Y (I\times Y)\cup_Y M_1\cong M_1\cup_Y M_1\cong M. \]
		To obtain a smooth metric on $N$, we choose a smooth function $k\colon [0,\varepsilon']\to[0,\infty)$ such that
		\begin{enumerate}
			\item $k$ is an odd function at $t=0$ with $k'(0)=1$ and $k'''(0)<0$,
			\item $k(\varepsilon')>0$ and all derivatives of $k$ vanish at $t=\varepsilon'$,
			\item $k''<0$ on $(0,\varepsilon')$.
		\end{enumerate}
		We then define the metric $g_N$ on $N$ by
		\[ g_N=\begin{cases}
			k(t)^2ds^2+dt^2+f(t)^2h_0,\quad & \text{on }I\times [0,\varepsilon']\times Y,\\
			k(\varepsilon')^2ds^2+g_0,\quad & \text{else}.
		\end{cases} \]
		Here $ds^2$ denotes the standard metric on $I$. Since $k(0)=0$, the metric $g_N$ is in fact a metric on
		\[ (I\times M_1)\cup_{I\times Y}(D^2_+\times Y), \]
		which is diffeomorphic to the space obtained from $N$ by smoothing the corners. Here $D^2_+=(\R\times[0,\infty))\cap D^2$ is a half-disc and we identify the interior face $(\R\times\{0\})\cap D^2$ of its boundary with $I$.
		
		By the boundary conditions of $f$ and $k$ at $t=0$, doubling the metric $g_N$ along its boundary results in a smooth metric on $(S^1\times M_1)\cup_{S^1\times Y}(D^2\times Y)$, see e.g.\ \cite[Proposition 1.4.7]{Pe16}. In particular, the metric $g_N$ itself is smooth and has totally geodesic boundary. Moreover, by Lemma \ref{L:doubly_warped_Ric}, the metric $g_N$ has non-negative Ricci curvature, and strictly positive Ricci curvature on $I\times[0,\varepsilon']\times Y$. Hence, by \cite[Proposition 2.15]{Re24a}, we can deform the metric $g_N$ into a metric (which we again denote by $g_N$) of strictly positive Ricci curvature and convex boundary, while leaving the induced metric on the boundary unchanged.
		
		The induced metric on the boundary is given by the double of the metric $g_0$. By applying the deformation $g_s$, we obtain a deformation through metrics of positive Ricci curvature to the double of the metric $g_1$, i.e.\ to $g$. Hence, by \cite[Theorem C]{Bu20}, we obtain a metric of positive Ricci curvature on $N$ with convex boundary isometric to $g$. In particular, by Lemma \ref{L:Ricci_cl_crit}, there exists $c>0$ so that $(M,c^2g)$ is Ricci closable.	
	\end{proof}
	
	We now apply Proposition \ref{P:Ricci_cl_crit} to the following $2\ell$-fold connected sum.
	
	\begin{proposition}\label{P:Ricci_clos_conn_sum}
		Let $M_1^n,\dots,M_\ell^n$ be closed manifolds that admit core metrics. Then there exists a metric $g$ of positive Ricci curvature on the space
		\[ M_1\#\dots\# M_\ell\#(-M_1)\#\dots\#(-M_\ell) \]
		such that for each summand there exists an isometric embedding of $\pm M_i\setminus {D^n}^\circ$ equipped with the corresponding core metric and such that $c^2g$ is Ricci closable for some $c>0$.
	\end{proposition}
	\begin{proof}
		We use Perelman's \textquotedblleft docking station\textquotedblright\ \cite{Pe97}, which, for any $\ell\in\N$ and $\nu>0$ sufficiently small is a metric of positive Ricci curvature on $S^n\setminus\sqcup_\ell {D^n}^\circ$ with round boundary components on which the second fundamental form is at least $-\nu$. This metric is the combination of the \textquotedblleft neck\textquotedblright\ of \cite[Section 2]{Pe97} and the \textquotedblleft ambient space\textquotedblright\ of \cite[Section 3]{Pe97}, see also \cite[Section 4]{Bu19}.
		
		More precisely, the \textquotedblleft ambient space\textquotedblright\ is a doubly warped product metric of positive sectional curvature on the sphere $S^n$ given by
		\[ dt^2+\cos(t)^2dx^2+R(t)^2ds_{n-1}^2 \]
		for $t\in[0,\tfrac{\pi}{2}]$ and $dx^2$ denotes the standard metric on $S^1$. The function $R$ is a smooth function which is odd at $t=0$ with $R'(0)=1$ and even at $t=\tfrac{\pi}{2}$, and satisfies $R''<0$. From this metric one now cuts out $\ell$ small disjoint geodesic balls along the circle $\{t=0\}$, and for a suitable choice of $R$ one can glue in $\ell$ copies of the \textquotedblleft neck metric\textquotedblright\ on $[0,1]\times S^{n-1}$, that transitions to the round metric and a sufficiently small second fundamental form.
		
		By choosing sufficiently small radii and placing all geodesic balls on one side of the circle, we can restrict this metric to the hemisphere $S^n_+$ defined by $[0,\tfrac{\pi}{2}]\times S^1_+\times S^{n-1}$, where $S^1_+\subseteq S^1$ is a half-circle, and we can still attach $\ell$ disjoint copies of the \textquotedblleft neck\textquotedblright. 
		
		By choosing $\nu$ sufficiently small, we can now glue $M_1\setminus{D^n}^\circ, \dots,M_\ell\setminus{D^n}^\circ$ to $S^n_+\setminus\sqcup_\ell{D^n}^\circ$ using Theorem~\ref{T:gluing}, which results in a metric of positive Ricci curvature on $(M_1\#\dots\# M_\ell)\setminus{D^n}^\circ$. Moreover, by the form of the original metric, doubling results in a smooth metric on
		\[ M_1\#\dots\# M_\ell\#(-M_1)\#\dots\#(-M_\ell), \]
		and interchanging $M_1\#\dots\# M_\ell$ and $-(M_1\#\dots\# M_\ell)$ defines an isometric $\Z/2$-action with fixed point set the hypersurface along which we have glued. Hence, the claim follows from Proposition \ref{P:Ricci_cl_crit}.
	\end{proof}
	
	\begin{lemma}\label{L:core_deform}
		Let $M^n$ be a closed manifold that admits a core metric $g$. Then there exists $\nu>0$ and a smooth  family $g_s$, $s\in(0,1]$, of Riemannian metrics on $M\setminus{D^n}^\circ$ with $g_1=g$ such that the following holds:
		\begin{enumerate}
			\item The volume and the Ricci curvatures of $g_s$ are bounded from below by positive constants that do not depend on $s$,
			\item For all $s\in(0,1]$ the induced metric on the boundary $S^{n-1}=\partial(M\setminus{D^n}^\circ)$ is the round metric of radius $1$ and the principal curvatures are bounded from below by $\nu$,
			\item As $s\to0$ the sequence $(M\setminus{D^n}^\circ,g_s)$ Gromov--Hausdorff converges to the disc $D^n$ equipped with a rotationally symmetric Riemannian metric with a singularity at the origin.
		\end{enumerate}
	\end{lemma}
	\begin{proof}
		The proof is an adaptation of \cite[Section 4.1]{CN13}. Let $g$ be a metric of positive Ricci curvature on $M\setminus {D^n}^\circ$ with round and convex boundary and let $\nu>0$ such that the principal curvatures at the boundary are all at least $2\nu$. For $s\in(0,\frac{\pi}{4\nu})$ we define the metric $h_s$ on $[s,\frac{\pi}{4\nu}]\times S^{n-1}$ by
		\[ h_s=dt^2+2\sin^2(\nu t)ds_{n-1}^2. \]
		By Lemma \ref{L:doubly_warped_Ric}, the Ricci curvatures of the metric $h_s$ are bounded from below by a positive constant independent of $s$. The boundary component $\{t=\frac{\pi}{4\nu}\}$ is round with principal curvatures given by $\nu$, while the boundary component $\{t=s\}$ is, after rescaling, round with principal curvatures given by $-\sqrt{2}\nu\cos(\nu s)> -2\nu$. Hence, we can glue the Riemannian manifolds 
		\[ (M\setminus{D^n}^\circ,2\sin^2(\nu s)g)\quad \text{ and }\quad ([s,\tfrac{\pi}{4\nu}]\times S^{n-1},h_s) \]
		using Theorem \ref{T:gluing} preserving the lower bounds on the Ricci curvatures. By the explicit form of the smoothing in Theorem \ref{T:gluing}, see \cite[Section 2]{BWW19}, we can arrange that the resulting family of metrics is smooth in $s$. This family is the required family $g_s$.
	\end{proof}
	
	\begin{proof}[Proof of Theorem \ref{T:core_tangent_cone}]
		We start with the metric $g$ on
		\[ X=M_1\#\dots\# M_\ell\#(-M_1)\#\dots\#(-M_\ell) \]
		constructed in Proposition \ref{P:Ricci_clos_conn_sum}. We set $M_{\ell+i}=-M_i$ for $i\in\{1,\dots,\ell\}$. For $\bar{s}=(s_1,\dots,s_{2\ell})\in(0,1]^{2\ell}$ we now define the metric $g_{\bar{s}}$ as the metric obtained from $g$ by replacing each core metric on $M_i$ by the metric $g_{s_i}$ constructed in Lemma \ref{L:core_deform} and smoothing the gluing area using Theorem \ref{T:gluing}. Then, after suitable rescaling, the family $\{g_{\bar{s}}\}$ satisfies the requirements of Theorem \ref{T:tang_cone_crit}, so that we obtain a non-collapsed Ricci limit space $(Y,d_Y,p)$ satisfying $\overline{\Omega}_{Y,p}=\overline{\{(X,g_{\bar{s}})\}}$
		
		Moreover, for each $i\in\{1,\dots,\ell\}$, the sequence of metrics $g_{\bar{s}}$ with $\bar{s}=(s,\dots,s,1,s,\dots,s)$, where the entry $1$ is at position $i$, converges to a metric $d_i$ on $M_i$ with $(2\ell-1)$ singularities as $s\to 0$. In particular, $(M_i,d_i)\in \overline{\Omega}_{Y,p}$.
	\end{proof}
	
	\begin{proof}[Proof of Corollary \ref{C:dim-4}]
		As explained in \cite{SY93}, if $M^4$ admits a Riemannian metric of positive scalar curvature, it is homeomorphic to either
		\[\#_\ell (S^2\times S^2)\quad\text{ or }\quad \#_k\C P^2\#_\ell(-\C P^2).\]
		For convenience we outline the argument. First note that, since $M$ is simply-connected, its second Stiefel--Whitney class equals its second Wu class and hence its intersection form is even if and only if it is spin. For the results on 4-manifolds and intersection forms we will use we refer to \cite[Section 1.2]{GS99}.
		
		Since in dimension $4$ the signature is a multiple of the $\hat{A}$-genus, it follows that $M$ has vanishing signature if it is spin. Hence, the intersection form of $M$ is isomorphic to that of $\#_{\ell}(S^2\times S^2)$ as they have the same rank, parity and signature. If $M$ is non-spin, its intersection form is odd, and hence its intersection form is equivalent to that of $\#_k\C P^2\#_{\ell}(-\C P^2)$ by Donaldson's theorem (see \cite{Do83}, \cite[Theorem 1.2.30]{GS99}). Hence, by Freedman's theorem (see \cite{Fr82}, \cite[Theorem 1.2.27]{GS99}), $M$ is homeomorphic to $\#_{\ell}(S^2\times S^2)$ when it is spin, and to $\#_k\C P^2\#_{\ell}(-\C P^2)$ when it is non-spin. By \cite{Bu19,Bu20a}, \cite{Re24a} all of these spaces admit core metrics, hence the first claim follows from Theorem \ref{T:core_tangent_cone}.
		
		For the second claim, assume that $M$ bounds a compact, oriented $5$-manifold. Then $M$ has vanishing signature. Further, since the signature is an oriented homeomorphism invariant, it follows that $M$ is homeomorphic to one of $\#_\ell (S^2\times S^2)$ and $\#_\ell\C P^2\#_\ell(-\C P^2)$. All these manifolds admit Ricci closable metrics by Proposition \ref{P:Ricci_clos_conn_sum} and Proposition \ref{P:SpSq} below.
	\end{proof}
	
	\section{Further examples of Ricci closable manifolds}\label{S:Ricci_clos_ex}
	
	In this section, we collect results to construct Ricci closable manifolds. The proofs directly follow from well-known construction methods for positive Ricci curvature. We therefore omit most of the proofs.
	
	\begin{proposition}\label{P:bundle}
		Let $E\xrightarrow{\pi}B$ be a fibre bundle with fibre $F$ and structure group $G$ such that $E$ is closed. Suppose the following:
		\begin{enumerate}
			\item $B$ admits a Riemannian metric $g_B$ of positive Ricci curvature,
			\item $F$ admits a $G$-invariant metric $g_F$ of positive Ricci curvature,
			\item There exists a compact Riemannian manifold $(\overline{F},\overline{g})$ of positive Ricci curvature with convex boundary isometric to $(F,g_F)$ such that the $G$-action on $F$ extends to an isometric action on $\overline{F}$.
		\end{enumerate}
		Then $E$ admits a Ricci closable metric $g_E$ such that $(E,g_E)\xrightarrow{\pi}(B,c^2g_B)$ is a Riemannian submersion with totally geodesic fibres isometric to ${c'}^2g_F$ for some $c,c'>0$.
	\end{proposition}
	This follows for example from the methods of \cite[Section 9.G]{Be87}, \cite[Theorem 2.7.3]{GW09}, \cite{RW23b} to lift metrics of positive Ricci curvature along fibre bundles.
	
	In particular, the assumptions of Proposition \ref{P:bundle} are satisfied when $\pi$ is a linear $S^p$-bundle with $p\geq 2$ whose base admits a Riemannian metric of positive Ricci curvature. Indeed, by setting $g_F=ds_p^2$ and $\overline{F}=D^{p+1}$ equipped with the induced metric of a geodesic ball in the round sphere of some suitable radius condition (3) is satisfied. In the case of a 1-dimensional fibre we have the following result:
	
	\begin{proposition}\label{P:S1-bundle}
		Let $E\xrightarrow{\pi}B$ be a principal $S^1$-bundle such that $B$ is closed and admits a Riemannian metric $g_B$ of positive Ricci curvature and $E$ has finite fundamental group. Then $E$ admits a Ricci closable metric $g_E$ such that $(E,g_E)\xrightarrow{\pi}(B,c^2g_B)$ is a Riemannian submersion for some $c>0$.
	\end{proposition}
	\begin{proof}
		As in Proposition \ref{P:bundle}, the corresponding linear $D^2$-bundle $\overline{E}$ has positive Ricci curvature when we equip it with a submersion metric with totally geodesic fibres equipped with the metric of a sufficiently small round hemisphere (see e.g.\ \cite[9.59 and 9.70]{Be87}). Here we can freely choose the principal connection. In particular, the boundary $E$ is totally geodesic as well. If we choose this metric to have harmonic curvature form, the induced metric on the boundary has non-negative Ricci curvature by \cite{BB78} and this metric can be deformed to have positive Ricci curvature by \cite{GPT98}. Hence, by the deformation of \cite[Proposition 2.15]{Re24a}, followed by a small deformation that makes the Ricci curvatures on $E$ positive, we obtain a metric on $\overline{E}$ with the required properties.
	\end{proof}
	Proposition \ref{P:S1-bundle} can for example be applied to connected sums
	\[L(m;1,\overset{n+1}{\cdots},1)\#_\ell (S^2\times S^{2n-1})  \]
	when $m$ is odd or $n$ is odd, and to
	\[L(m;1,\overset{n+1}{\cdots},1)\#_\ell (S^2\ttimes S^{2n-1})  \]
	when $m$ and $n$ are even, where $S^2\ttimes S^{2n-1}$ denotes the total space of the unique non-trivial linear $S^{2n-1}$-bundle over $S^2$ and $\ell\in\N_0$. Indeed, by \cite[Theorems A and B]{GR23}, these manifolds are total spaces of principal $S^1$-bundles over $\#_{\ell+1}\C P^{n}$.
	
	When $m=1$ this shows that the manifold $\#_\ell(S^2\times S^{2n-1})$ admits a Ricci closable metric for all $\ell\geq 0$. By Proposition \ref{P:Ricci_clos_conn_sum} the same holds for connected sums of other products of spheres when $\ell$ is even. We now extend this to odd values of $\ell$ and arbitrary dimensions of the spheres involved.	
	\begin{proposition}\label{P:SpSq}
		For $p,q\geq 2$, $\ell\geq 0$, there exists a Ricci closable metric on the connected sum $\#_\ell (S^p\times S^q)$.
	\end{proposition}
	\begin{proof}
		If $\ell$ is even, this follows from Proposition \ref{P:Ricci_clos_conn_sum} since there is an orientation-preserving diffeomorphism $S^p\times S^q\cong-(S^p\times S^q)$. Hence, it remains to consider the case where $\ell$ is odd.
		
		For that we use the construction of Sha--Yang \cite{SY91} (see also \cite[Section 5.5]{Re22} for a discussion of this result), where $\#_\ell (S^p\times S^q)$ is constructed as the manifold obtained from $(\ell+1)$ surgeries on the second factor of $S^{p+1}\times S^{q-1}$, i.e.\
		\[ \#_\ell(S^p\times S^q)\cong \left(\left(S^{p+1}\setminus\bigsqcup_{\ell+1}{D^{p+1}}^\circ\right)\times S^{q-1}\right)\cup_{\sqcup_{\ell+1}(S^p\times S^{q-1})}\bigsqcup_{\ell+1}(S^{p}\times D^q). \]
		The starting point for the construction is the round metric on $S^{p+1}$, from which $(\ell+1)$ pairwise disjoint geodesic balls are removed. After taking the product with $S^{q-1}$ equipped with a round metric of sufficiently small radius, $(\ell+1)$ copies of $(S^p\times D^q)$, equipped with suitable metrics that glue smoothly with the metric on $(S^{p+1}\setminus\sqcup_{\ell+1}{D^{p+1}}^\circ)\times S^{q-1}$, are attached.
		
		Since we can freely choose the positions and radii of the geodesic balls in $S^{p+1}$, we can arrange that, since $(\ell+1)$ is even, each one gets mapped to another one under the isometric $\Z/2$-action on $S^{p+1}$ given by reflection along the equator. In this way we obtain a metric of non-negative Ricci curvature, and strictly positive Ricci curvature if $q>2$, on the glued space that is invariant under a $\Z/2$-action, whose fixed point set is given by $S^p\times S^{q-1}$ equipped with the product of two round metrics. For $q=2$ one additionally applies the deformation results of \cite{Eh76}, which, as explained in \cite[p.\ 20]{Eh76}, can be arranged to preserve the $\Z/2$-invariance of the metric. Hence, we can apply Proposition \ref{P:Ricci_cl_crit} to obtain a Ricci closable metric on $\#_\ell(S^p\times S^q)$.
	\end{proof}
	
	Finally, we consider Ricci closable metrics with large isometry group.
	
	\begin{proposition}\label{P:hom_space}
		Let $G/H$ be a homogeneous space with compact Lie groups $H\subseteq G$ such that the fundamental group of $G/H$ is finite. Suppose that there exists a subgroup $H\subseteq K\subseteq G$ such that $G/K$ has finite fundamental group and $K/H$ is diffeomorphic to a sphere of dimension $d\geq 1$. Then $G/H$ admits a homogeneous metric that is Ricci closable.
	\end{proposition}
	
	This follows from \cite{GZ02}, since under these assumptions the manifold $G/H$ is the boundary of the cohomogeneity one manifold $G\times_K D^{d+1}$, where $D^{d+1}$ is identified with the cone over $K/H$.
	
	For example, Proposition \ref{P:hom_space} can be applied to the quotient of $S^{4n-1}\subseteq\Quat^n$ by the standard action of the binary dihedral group $Dic_m\subseteq S^3=Sp(1)$. Indeed, this space is the homogeneous space $G/H=Sp(n)/Sp(n-1)Dic_m$, and the group $K$ is given by $Sp(n-1)Pin(2)$.
	
	\begin{proposition}\label{P:cohom_one}
		Let $M$ be a closed manifold that admits a cohomogeneity one action of a compact Lie group $G$ such that the orbit space $M/G$ is homeomorphic to the interval $[-1,1]$ and such that both $M$ and its principal orbits have finite fundamental group. Suppose that the associated group diagram $H\subseteq K_\pm\subseteq G$ satisfies $K_+=K_-$. Then $M$ admits a $G$-invariant Riemannian metric that is Ricci closable.
	\end{proposition}
	
	See e.g.\ \cite[Section 1]{GZ02} for an introduction to cohomogeneity one manifolds. The proof then directly follows from the construction in \cite{GZ02} together with Proposition \ref{P:Ricci_cl_crit} since the $\Z/2$-action that interchanges the two halves $\pi^{-1}([-1,0])$ and $\pi^{-1}([0,1])$, where $\pi\colon M\to M/G$ denotes the projection, is isometric and has fixed point set $G/H$.
	
	For example, Proposition \ref{P:cohom_one} can be applied to $\C P^{n}\#(-\C P^n)$, which has group diagram $H\subseteq K_\pm\subseteq G$ with $H=U(n-1)$, $K_+=K_-= U(n-1)U(1)$ and $G=U(n)$. In fact, since the principal orbits have codimension two, it follows from the construction in \cite{GZ00} that we can construct a Ricci closable metric that has non-negative sectional curvature.
	
	\appendix
	
	\section{Tangent cones of collapsed Ricci limit spaces} \label{S:collapsed}
	
	In the context of surgery on manifolds of positive Ricci curvature, Sha--Yang \cite{SY91} constructed a metric on $S^n\times D^m$ of $\Ric\geq0$ that is close to the product $C(S^n,ds_n^2)\times (S^{m-1},R^2ds_{m-1}^2)$ for some $R>0$ near the boundary. A consequence of this construction is the following theorem. We include its proof for convenience.
	
	\begin{theorem}\label{T:tangent_cone_collapsed}
		Let $(M^n,g)$ be a closed Riemannian manifold with $\Ric\geq (n-1)$. Then, for any $m\geq 2$ there exists a complete Riemannian metric of $\Ric\geq 0$ on $M\times \R^m$ with asymptotic cone given by $C(M,g)$.
	\end{theorem}
	
	\begin{proof}
		Let $\alpha=2\frac{n-1}{m}$ and let $f\colon[0,\infty)\to\R$ be the unique solution of the initial value problem
		\begin{align*}
			f''&=\frac{\alpha}{2}f^{-\alpha-1},\\
			f(0)&=1,\\
			f'(0)&=0.
		\end{align*}
		Further, let $h\colon[0,\infty)\to\R$ be the function
		\[ h=\frac{2}{\alpha}f' \]
		and define the metric $g_{f,h}$ on $[0,\infty)\times S^{m-1}\times M$ by
		\[ g_{f,h}=dt^2+h(t)^2ds_{m-1}^2+f(t)^2g. \]
		By the initial conditions of $f$ we have $h(0)=0$ and $h'(0)=1$, and by the defining equation of $f$ we obtain inductively that at $t=0$ the function $h$ is odd and $f$ is even. Therefore, the metric $g_{f,h}$ defines a smooth metric on $\R^m\times M$.
		
		By Lemma \ref{L:doubly_warped_Ric}, the Ricci curvatures of the metric $g_{f,h}$ are given as follows:
		\begin{align*}
			\Ric(\partial_t,\partial_t)&=-(m-1)\frac{h''}{h}-n\frac{f''}{f},\\
			\Ric(\tfrac{u}{h},\tfrac{u}{h})&=-\frac{h''}{h}+(m-2)\frac{1-{h'}^2}{h^2}-n\frac{h'f'}{hf},\\
			\Ric(\tfrac{v}{f},\tfrac{v}{f})&=-\frac{f''}{f}+\frac{\Ric^g(v,v)-(n-1){f'}^2}{f^2}-(m-1)\frac{h'f'}{hf}\\
			&\geq -\frac{f''}{f}+(n-1)\frac{1-{f'}^2}{f^2}-(m-1)\frac{h'f'}{hf},\\
			\Ric(\partial_t,u)&=\Ric(\partial_t,v)=\Ric(u,v)=0,
		\end{align*}
		where $u\in TS^{m-1}$ and $v\in TM$ are unit vectors with respect to $ds_{m-1}^2$ and $g$, respectively.
		
		Integrating the equation $f''f'=\frac{\alpha}{2}f^{-\alpha-1}f'$ now shows that
		\[ {f'}^2=1-f^{-\alpha}. \]
		Further, we have
		\[ \frac{1-{h'}^2}{h^2}=\frac{\alpha^2}{4}\frac{1-f^{-2\alpha-2}}{1-f^{-\alpha}}\geq \frac{\alpha^2}{4}f^{-\alpha-2},\quad \frac{h''}{h}=-\frac{\alpha(\alpha+1)}{2}f^{-\alpha-2},\quad \frac{h'f'}{hf}=\frac{\alpha}{2}f^{-\alpha-2}. \]
		Using this a calculation shows that all Ricci curvatures are non-negative.
		
		Finally, since $f''>0$, we have $f(t)\to\infty$ as $t\to\infty$, and therefore $f'(t)\to1$ and $h(t)\to\frac{2}{\alpha}$ as $t\to\infty$. This shows that $R^2g_{f,h}$ converges to the cone $C(M,g)$ as $R\to 0$.
	\end{proof}
	
	\begin{remark}
		Similar arguments using the generalisations of \cite{SY91} given in \cite{Wr98} and \cite{Re23} show that $M\times \R^m$ in Theorem \ref{T:tangent_cone_collapsed} can be replaced by $M\times \mathrm{int}(N)$, where $N$ is a compact manifold with boundary that admits a Riemannian metric of positive Ricci curvature such that on the boundary the second fundamental form is non-negative and the induced metric has positive Ricci curvature (e.g.\ if $N=N'\setminus{D^m}^\circ$ and $N'$ admits a core metric).
	\end{remark}

	\bibliographystyle{plainurl}
	\bibliography{References}

\begin{thebibliography}{10}

\bibitem{BH22}
Christian B\"{a}r and Bernhard Hanke.
\newblock Local flexibility for open partial differential relations.
\newblock {\em Comm. Pure Appl. Math.}, 75(6):1377--1415, 2022.
\newblock \href {https://doi.org/10.1002/cpa.21982}
  {\path{doi:10.1002/cpa.21982}}.

\bibitem{BB78}
Lionel B\'{e}rard-Bergery.
\newblock Certains fibr\'{e}s \`a courbure de {R}icci positive.
\newblock {\em C. R. Acad. Sci. Paris S\'{e}r. A-B}, 286(20):A929--A931, 1978.

\bibitem{Be87}
Arthur~L. Besse.
\newblock {\em Einstein manifolds}, volume~10 of {\em Ergebnisse der Mathematik
  und ihrer Grenzgebiete (3) [Results in Mathematics and Related Areas (3)]}.
\newblock Springer-Verlag, Berlin, 1987.

\bibitem{BWW19}
Boris Botvinnik, Mark~G. Walsh, and David~J. Wraith.
\newblock Homotopy groups of the observer moduli space of {R}icci positive
  metrics.
\newblock {\em Geom. Topol.}, 23(6):3003--3040, 2019.
\newblock \href {https://doi.org/10.2140/gt.2019.23.3003}
  {\path{doi:10.2140/gt.2019.23.3003}}.

\bibitem{BPS24}
Elia Bruè, Alessandro Pigati, and Daniele Semola.
\newblock Topological regularity and stability of noncollapsed spaces with
  ricci curvature bounded below.
\newblock {\em arXiv e-prints}, 2024.
\newblock \href {https://arxiv.org/abs/2405.03839} {\path{arXiv:2405.03839}}.

\bibitem{Bu19a}
Bradley~Lewis Burdick.
\newblock {\em Metrics of {P}ositive {R}icci {C}urvature on {C}onnected {S}ums:
  {P}rojective {S}paces, {P}roducts, and {P}lumbings}.
\newblock ProQuest LLC, Ann Arbor, MI, 2019.
\newblock Thesis (Ph.D.)--University of Oregon.
\newblock URL:
  \url{http://gateway.proquest.com/openurl?url_ver=Z39.88-2004&rft_val_fmt=info:ofi/fmt:kev:mtx:dissertation&res_dat=xri:pqm&rft_dat=xri:pqdiss:13898429}.

\bibitem{Bu19}
Bradley~Lewis Burdick.
\newblock Ricci-positive metrics on connected sums of projective spaces.
\newblock {\em Differential Geom. Appl.}, 62:212--233, 2019.
\newblock \href {https://doi.org/10.1016/j.difgeo.2018.11.005}
  {\path{doi:10.1016/j.difgeo.2018.11.005}}.

\bibitem{Bu20}
Bradley~Lewis Burdick.
\newblock Metrics of positive {R}icci curvature on the connected sums of
  products with arbitrarily many spheres.
\newblock {\em Ann. Global Anal. Geom.}, 58(4):433--476, 2020.
\newblock \href {https://doi.org/10.1007/s10455-020-09732-7}
  {\path{doi:10.1007/s10455-020-09732-7}}.

\bibitem{Bu20a}
Bradley~Lewis Burdick.
\newblock The space of positive ricci curvature metrics on spin manifolds.
\newblock {\em arXiv e-prints}, 2020.
\newblock \href {https://arxiv.org/abs/2009.06199} {\path{arXiv:2009.06199}}.

\bibitem{Ch01}
Jeff Cheeger.
\newblock {\em Degeneration of {R}iemannian metrics under {R}icci curvature
  bounds}.
\newblock Lezioni Fermiane. [Fermi Lectures]. Scuola Normale Superiore, Pisa,
  2001.

\bibitem{CC97}
Jeff Cheeger and Tobias~H. Colding.
\newblock On the structure of spaces with {R}icci curvature bounded below. {I}.
\newblock {\em J. Differential Geom.}, 46(3):406--480, 1997.
\newblock URL: \url{http://projecteuclid.org/euclid.jdg/1214459974}.

\bibitem{CC00}
Jeff Cheeger and Tobias~H. Colding.
\newblock On the structure of spaces with {R}icci curvature bounded below.
  {II}.
\newblock {\em J. Differential Geom.}, 54(1):13--35, 2000.
\newblock URL: \url{http://projecteuclid.org/euclid.jdg/1214342145}.

\bibitem{CC00a}
Jeff Cheeger and Tobias~H. Colding.
\newblock On the structure of spaces with {R}icci curvature bounded below.
  {III}.
\newblock {\em J. Differential Geom.}, 54(1):37--74, 2000.
\newblock URL: \url{http://projecteuclid.org/euclid.jdg/1214342146}.

\bibitem{CJN21}
Jeff Cheeger, Wenshuai Jiang, and Aaron Naber.
\newblock Rectifiability of singular sets of noncollapsed limit spaces with
  {R}icci curvature bounded below.
\newblock {\em Ann. of Math. (2)}, 193(2):407--538, 2021.
\newblock \href {https://doi.org/10.4007/annals.2021.193.2.2}
  {\path{doi:10.4007/annals.2021.193.2.2}}.

\bibitem{CN13a}
Jeff Cheeger and Aaron Naber.
\newblock Lower bounds on {R}icci curvature and quantitative behavior of
  singular sets.
\newblock {\em Invent. Math.}, 191(2):321--339, 2013.
\newblock \href {https://doi.org/10.1007/s00222-012-0394-3}
  {\path{doi:10.1007/s00222-012-0394-3}}.

\bibitem{CN15}
Jeff Cheeger and Aaron Naber.
\newblock Regularity of {E}instein manifolds and the codimension 4 conjecture.
\newblock {\em Ann. of Math. (2)}, 182(3):1093--1165, 2015.
\newblock \href {https://doi.org/10.4007/annals.2015.182.3.5}
  {\path{doi:10.4007/annals.2015.182.3.5}}.

\bibitem{Co97}
Tobias~H. Colding.
\newblock Ricci curvature and volume convergence.
\newblock {\em Ann. of Math. (2)}, 145(3):477--501, 1997.
\newblock \href {https://doi.org/10.2307/2951841} {\path{doi:10.2307/2951841}}.

\bibitem{CN13}
Tobias~Holck Colding and Aaron Naber.
\newblock Characterization of tangent cones of noncollapsed limits with lower
  {R}icci bounds and applications.
\newblock {\em Geom. Funct. Anal.}, 23(1):134--148, 2013.
\newblock \href {https://doi.org/10.1007/s00039-012-0202-7}
  {\path{doi:10.1007/s00039-012-0202-7}}.

\bibitem{Do83}
S.~K. Donaldson.
\newblock An application of gauge theory to four-dimensional topology.
\newblock {\em J. Differential Geom.}, 18(2):279--315, 1983.
\newblock URL: \url{http://projecteuclid.org/euclid.jdg/1214437665}.

\bibitem{Eh76}
Paul Ehrlich.
\newblock Metric deformations of curvature. {I}. {L}ocal convex deformations.
\newblock {\em Geometriae Dedicata}, 5(1):1--23, 1976.
\newblock \href {https://doi.org/10.1007/BF00148134}
  {\path{doi:10.1007/BF00148134}}.

\bibitem{Fr82}
Michael~Hartley Freedman.
\newblock The topology of four-dimensional manifolds.
\newblock {\em J. Differential Geometry}, 17(3):357--453, 1982.
\newblock URL: \url{http://projecteuclid.org/euclid.jdg/1214437136}.

\bibitem{GR23}
Fernando Galaz-García and Philipp Reiser.
\newblock Free torus actions and twisted suspensions.
\newblock {\em arXiv e-prints}, 2023.
\newblock \href {https://arxiv.org/abs/2305.06068} {\path{arXiv:2305.06068}}.

\bibitem{GPT98}
Peter~B. Gilkey, JeongHyeong Park, and Wilderich Tuschmann.
\newblock Invariant metrics of positive {R}icci curvature on principal bundles.
\newblock {\em Math. Z.}, 227(3):455--463, 1998.
\newblock \href {https://doi.org/10.1007/PL00004385}
  {\path{doi:10.1007/PL00004385}}.

\bibitem{GS99}
Robert~E. Gompf and Andr\'as~I. Stipsicz.
\newblock {\em {$4$}-manifolds and {K}irby calculus}, volume~20 of {\em
  Graduate Studies in Mathematics}.
\newblock American Mathematical Society, Providence, RI, 1999.
\newblock \href {https://doi.org/10.1090/gsm/020} {\path{doi:10.1090/gsm/020}}.

\bibitem{GW09}
Detlef Gromoll and Gerard Walschap.
\newblock {\em Metric foliations and curvature}, volume 268 of {\em Progress in
  Mathematics}.
\newblock Birkh\"{a}user Verlag, Basel, 2009.
\newblock \href {https://doi.org/10.1007/978-3-7643-8715-0}
  {\path{doi:10.1007/978-3-7643-8715-0}}.

\bibitem{GZ00}
Karsten Grove and Wolfgang Ziller.
\newblock Curvature and symmetry of {M}ilnor spheres.
\newblock {\em Ann. of Math. (2)}, 152(1):331--367, 2000.
\newblock \href {https://doi.org/10.2307/2661385} {\path{doi:10.2307/2661385}}.

\bibitem{GZ02}
Karsten Grove and Wolfgang Ziller.
\newblock Cohomogeneity one manifolds with positive {R}icci curvature.
\newblock {\em Invent. Math.}, 149(3):619--646, 2002.

\bibitem{Ho17}
Shouhei Honda.
\newblock Ricci curvature and orientability.
\newblock {\em Calc. Var. Partial Differential Equations}, 56(6):Paper No. 174,
  47, 2017.
\newblock \href {https://doi.org/10.1007/s00526-017-1258-x}
  {\path{doi:10.1007/s00526-017-1258-x}}.

\bibitem{Ka02}
V.~Kapovitch.
\newblock Regularity of limits of noncollapsing sequences of manifolds.
\newblock {\em Geom. Funct. Anal.}, 12(1):121--137, 2002.
\newblock \href {https://doi.org/10.1007/s00039-002-8240-1}
  {\path{doi:10.1007/s00039-002-8240-1}}.

\bibitem{Ka07}
Vitali Kapovitch.
\newblock Perelman's stability theorem.
\newblock In {\em Surveys in differential geometry. {V}ol. {XI}}, volume~11 of
  {\em Surv. Differ. Geom.}, pages 103--136. Int. Press, Somerville, MA, 2007.
\newblock \href {https://doi.org/10.4310/SDG.2006.v11.n1.a5}
  {\path{doi:10.4310/SDG.2006.v11.n1.a5}}.

\bibitem{Ke15}
Christian Ketterer.
\newblock Cones over metric measure spaces and the maximal diameter theorem.
\newblock {\em J. Math. Pures Appl. (9)}, 103(5):1228--1275, 2015.
\newblock \href {https://doi.org/10.1016/j.matpur.2014.10.011}
  {\path{doi:10.1016/j.matpur.2014.10.011}}.

\bibitem{La70}
H.~Blaine Lawson, Jr.
\newblock The unknottedness of minimal embeddings.
\newblock {\em Invent. Math.}, 11:183--187, 1970.
\newblock \href {https://doi.org/10.1007/BF01404649}
  {\path{doi:10.1007/BF01404649}}.

\bibitem{Pe91}
Grigori Perelman.
\newblock Alexandrov spaces with curvatures bounded from below ii.
\newblock {\em preprint}, 1991.

\bibitem{Pe97a}
Grigori Perelman.
\newblock A complete {R}iemannian manifold of positive {R}icci curvature with
  {E}uclidean volume growth and nonunique asymptotic cone.
\newblock In {\em Comparison geometry ({B}erkeley, {CA}, 1993--94)}, volume~30
  of {\em Math. Sci. Res. Inst. Publ.}, pages 165--166. Cambridge Univ. Press,
  Cambridge, 1997.

\bibitem{Pe97}
Grigori Perelman.
\newblock Construction of manifolds of positive {R}icci curvature with big
  volume and large {B}etti numbers.
\newblock In {\em Comparison geometry ({B}erkeley, {CA}, 1993--94)}, volume~30
  of {\em Math. Sci. Res. Inst. Publ.}, pages 157--163. Cambridge Univ. Press,
  Cambridge, 1997.

\bibitem{Pe16}
Peter Petersen.
\newblock {\em Riemannian geometry}, volume 171 of {\em Graduate Texts in
  Mathematics}.
\newblock Springer, Cham, third edition, 2016.
\newblock \href {https://doi.org/10.1007/978-3-319-26654-1}
  {\path{doi:10.1007/978-3-319-26654-1}}.

\bibitem{Re22}
Philipp Reiser.
\newblock {\em Generalized Surgery on Riemannian Manifolds of Positive Ricci
  Curvature}.
\newblock PhD thesis, Karlsruher Institut für Technologie (KIT), 2022.
\newblock \href {https://doi.org/10.5445/IR/1000150280}
  {\path{doi:10.5445/IR/1000150280}}.

\bibitem{Re23}
Philipp Reiser.
\newblock Generalized surgery on {R}iemannian manifolds of positive {R}icci
  curvature.
\newblock {\em Trans. Amer. Math. Soc.}, 376(5):3397--3418, 2023.
\newblock \href {https://doi.org/10.1090/tran/8789}
  {\path{doi:10.1090/tran/8789}}.

\bibitem{Re24a}
Philipp Reiser.
\newblock Positive ricci curvature on connected sums of fibre bundles.
\newblock {\em arXiv e-prints}, 2024.
\newblock \href {https://arxiv.org/abs/2406.02274} {\path{arXiv:2406.02274}}.

\bibitem{RW23b}
Philipp Reiser and David~J. Wraith.
\newblock Positive intermediate ricci curvature on fibre bundles.
\newblock {\em arXiv e-prints}, 2023.
\newblock \href {https://arxiv.org/abs/2211.14610} {\path{arXiv:2211.14610}}.

\bibitem{SY91}
Ji-Ping Sha and DaGang Yang.
\newblock Positive {R}icci curvature on the connected sums of {$S^n\times
  S^m$}.
\newblock {\em J. Differential Geom.}, 33(1):127--137, 1991.
\newblock URL: \url{http://projecteuclid.org/euclid.jdg/1214446032}.

\bibitem{SY93}
Ji-Ping Sha and DaGang Yang.
\newblock Positive {R}icci curvature on compact simply connected
  {$4$}-manifolds.
\newblock In {\em Differential geometry: {R}iemannian geometry ({L}os
  {A}ngeles, {CA}, 1990)}, volume 54, Part 3 of {\em Proc. Sympos. Pure Math.},
  pages 529--538. Amer. Math. Soc., Providence, RI, 1993.
\newblock \href {https://doi.org/10.1090/pspum/054.3/1216644}
  {\path{doi:10.1090/pspum/054.3/1216644}}.

\bibitem{ST21}
Miles Simon and Peter~M. Topping.
\newblock Local mollification of {R}iemannian metrics using {R}icci flow, and
  {R}icci limit spaces.
\newblock {\em Geom. Topol.}, 25(2):913--948, 2021.
\newblock \href {https://doi.org/10.2140/gt.2021.25.913}
  {\path{doi:10.2140/gt.2021.25.913}}.

\bibitem{Wr98}
David~J. Wraith.
\newblock Surgery on {R}icci positive manifolds.
\newblock {\em J. Reine Angew. Math.}, 501:99--113, 1998.
\newblock \href {https://doi.org/10.1515/crll.1998.082}
  {\path{doi:10.1515/crll.1998.082}}.

\bibitem{Wr02}
David~J. Wraith.
\newblock Deforming {R}icci positive metrics.
\newblock {\em Tokyo J. Math.}, 25(1):181--189, 2002.
\newblock \href {https://doi.org/10.3836/tjm/1244208944}
  {\path{doi:10.3836/tjm/1244208944}}.

\bibitem{Zh24}
Shengxuan Zhou.
\newblock A family of $4$-manifolds with nonnegative ricci curvature and
  prescribed asymptotic cone.
\newblock {\em arXiv e-prints}, 2024.
\newblock \href {https://arxiv.org/abs/2406.02279} {\path{arXiv:2406.02279}}.

\end{thebibliography}

\end{document}